\documentclass[10pt,letterpaper,reqno]{amsart}
\usepackage{amssymb,a4wide,amsthm,amsmath,amsfonts,xcolor}
\usepackage{mathrsfs}
\usepackage{bbm}

\usepackage[pagewise]{lineno}

\usepackage{verbatim}
\usepackage[colorlinks,linkcolor={blue},
citecolor={blue},urlcolor={red},]{hyperref}
\usepackage{hyperref}
\usepackage[utf8]{inputenc}
\usepackage[T1]{fontenc}

\numberwithin{equation}{section}

\newtheorem{theorem}{Theorem}[section]
\newtheorem{lemma}[theorem]{Lemma}

\newtheorem{assumption}[theorem]{Assumption}

\newtheorem{definition}[theorem]{Definition}
\newtheorem{main result}{Main Result}

\newcommand\coma[1]{{\color{red} {#1}}}
\newcommand\dela[1]{}

\def\m{m^{\prime}}

\def\l{\left}
\def\r{\right}

\def\p{\prime}

\title[Relaxed optimal control for stochastic LLB equation with jump noise]{Relaxed optimal control for the stochastic Landau-Lifshitz-Bloch equation with Jump Noise}
\author{Soham Gokhale}
\address{Department of Applied Sciences,\\
Symbiosis Institute of Technology, Symbiosis International (Deemed University) (SIU), Pune}
\email{soham.gokhale@sitpune.edu.in}



\thanks{email : soham.gokhale@sitpune.edu.in}

\allowdisplaybreaks

\begin{document}
	\begin{abstract}
		We study the stochastic Landau-Lifshitz-Bloch equation perturbed by pure jump noise. In order to understand and also control the fluctuations and jumps observed in the hysteresis loop, we add noise and an external control to the effective field. We consider the control operator as one that can depend, even nonlinearly on both the control and the associated solution process. We reformulate the problem as a relaxed control problem by using random Young measures and show that the relaxed problem admits an optimal control. The proof employs the Aldous condition for establishing tightness of laws and a modified version of the Skorohod Theorem for obtaining subconvergence of laws.
	\end{abstract}
	\maketitle
	
	\section{Introduction}
	Weiss, see \cite{brown1963micromagnetics} initiated the study of ferromagnetism. Landau and Lifshitz \cite{Landau+Lifshitz_1935_TTheoryOf_MagneticPermeability_Ferromagnetic} and Gilbert \cite{Gilbert} further developed the study.	
	For temperatures $\mathbb{T}$ below the Curie temperature $\mathbb{T}_c$, the magnetization can be modelled by the Landau-Lifshitz-Gilbert (LLG) equation. It assumes that the length of the magnetization vector remains constant. But this does not hold true for higher temperatures, for example in heat assisted magnetic recording (HAMR). Garanin \cite{Garanin_1991_Generalized_EquationOfMotion_Ferromagnet,garanin1997fokker,garanin2004thermal} developed a thermodynamically consistent approach, the Landau-Lifshitz-Bloch (LLB) equation, which is valid for temperatures both below and above the Curie temperature. It essentially interpolates between the LLG at low temperatures and the Ginzburg-Landau theory of phase transitions. Let $\mathcal{O}\subset \mathbb{R}^d,d=1,2,3$ denote the domain with $T$ denoting the terminal time. The LLB equation is given by the following.
	\begin{align}\label{deterministic SLLBE}
		\begin{cases}
			&\frac{\partial m}{\partial t} =  \gamma m \times H_{\text{eff}} + L_1\frac{1}{|m|_{\mathbb{R}^3}^2}(m \cdot H_{\text{eff}})m - L_2\frac{1}{|m|_{\mathbb{R}^3}^2} m\times (m\times H_{\text{eff}}),\\
			&\frac{\partial m}{\partial \eta}(t,x) =  0,\ t>0,\ x\in \partial\mathcal{O}, \\
			&m(0) =  m_0.
		\end{cases}
	\end{align}
	Here, $\gamma$ denotes the damping parameter and $L_1, L_2$ are the longitudinal and transverse damping parameters respectively.	
	The effective field $H_{\text{eff}}$ is given by
	\begin{equation}\label{Definition of H eff}
		H_{\text{eff}} = \Delta m - \frac{1}{\mathcal{X}_{||}}\l( 1 + \frac{3}{5}\frac{\mathbb{T}}{\mathbb{T}-\mathbb{T}_c}\l|m\r|_{\mathbb{R}^3}^2 \r) m.
	\end{equation}	
	where $\mathcal{X}_{||}$ is the longitudinal susceptibility. On the right hand side of \eqref{Definition of H eff}, the first term denotes the exchange field, the second denotes the entropy correction field.
	
		Using the identity
	\begin{align*}
		a \times (b \times c) = b(a \cdot c) - c(a \cdot b)\ \text{for}\ a,b,c\in\mathbb{R}^3,
	\end{align*}
gives
 \begin{align}\label{eqn triple product formula}
		m \times (m \times H_{\text{eff}}) = (m \cdot H_{\text{eff}})m - H_{\text{eff}} \l|m\r|_{\mathbb{R}^3}^2.
	\end{align}
	For temperature above the Curie temperature, we have $L_1 = L_2 = \kappa_1$ (say).	Using the preceeding discussion, along with \eqref{Definition of H eff} and \eqref{eqn triple product formula}, we can write  equation \eqref{deterministic SLLBE} as
	\begin{align}
		\nonumber \frac{\partial m}{\partial t} = \, &  \gamma \, m \times \Delta m -    \frac{\gamma}{\mathcal{X}_{||}}\l( 1 + \frac{3}{5}\frac{\mathbb{T}}{\mathbb{T}-\mathbb{T}_c}\l|m\r|_{\mathbb{R}^3}^2 \r) \l( m \times m \r) + \kappa_1 \Delta m \\
		& - \frac{\kappa_1}{\mathcal{X}_{||}}\l( 1 + \frac{3}{5}\frac{\mathbb{T}}{\mathbb{T}-\mathbb{T}_c}\l|m\r|_{\mathbb{R}^3}^2 \r) m  .
	\end{align}
	Setting $\mu = \frac{3}{5}\frac{\mathbb{T}}{\mathbb{T}-\mathbb{T}_c} $ and $\kappa =  \frac{\kappa_1}{\chi_{||}}$ consequences in the following equality.
	\begin{align}\label{LLB equation with constants}
		\frac{\partial m}{\partial t} = &   \kappa_1 \Delta m + \gamma \, m \times \Delta m -  
		\kappa\l( 1 + \mu \l|m\r|_{\mathbb{R}^3}^2 \r) m 
		.
	\end{align}
	
		Le \cite{LE_Deterministic_LLBE} proved the existence of a weak solution in a bounded domain for $d=1,2,3$.	
	The reader can refer to \cite{Ayouch+Benmouane+Essoufi_2022_RegularSolution_LLB,Hadamache+Hamroun_2020_LargeSolution_LLBE,Li+Guo+Zeng_2021_SmoothSolution_LLBE,Pu+Yang_2022_GlobalSmoothSolutions_LLBE,Wang+EtAl_SmoothSolutionLLBE} and references within for some recent developments.
	The (deterministic) LLB equation turns out to be insufficient, for example, to capture the dispersion of individual trajectories at high temperatures. 
	Hence, Brown  \cite{brown1963micromagnetics,Brown_Thermal_Fluctuations}  modified the LLB equation in order to incorporate random fluctuations and to describe noise induced transitions between equilibrium states of the ferromagnet. The reader can see \cite{Evans_etal_2012_Stochastic_Form_LLBE_Physics,garanin2004thermal} for discussions on stochastic form of LLB equation. Brzezniak, Goldys and Le in \cite{ZB+BG+Le_SLLBE} proved the existence of a weak martingale solution, along with the existence of invariant measures to the stochastic LLB equation. Jiang, Ju and Wang in \cite{Jiang+Ju+Wang_MartingaleWeakSolnSLLBE} show the existence of a weak martingale solution to the stochastic LLB equation. Qiu, Tang and Wang in \cite{Qiu+Tang_Wang_AsymptoticBehaviourSLLBE} establish a large deviations principle, along with a central limit theorem for the stochastic LLB equation for dimension $1$. The authors in \cite{UM+SG_2022_SLLBE_WongZakai} prove the existence of a pathwise unique solution along with proving Wong-Zakai type approximations for the stochastic LLB equations for $d=1,2$.

For example, magnetic nanostructures are extremely susciptible to heat excitations. Since data writing, reading and storing is among the primary application areas, we should take into account how the phenomenon of thermal excitation may be involved in these, either as a consequence or initiating magnetization switching (for example thermally induced magnetization switching).
There are more than one reasons to conduct this study. Another reason is as follows. Impurities and material defects could exist within a magnetic domain, which is traversed by domain walls. Consequently, these defects bind the domain wall, making it non-free. These flaws may serve as small-scale potential obstacles. There may be a snapping effect as an external field is used to move the wall, which could result in a jump in the hysterisis loop.
 The hysteresis loop is therefore distinguished by a series of jumps, known as Barkhausen jumps. Mayergoyz, \textit{et. al.}, in \cite{mayergoyz_2011_LandauLifshitzMagnetizationDynamics_JumpNoise,mayergoyz_2011_MagnetizationDynamics_JumpNoise} introduced jump noise process into magnetization dynamics equations in order to account for the random thermal effects. The exact mechanism of the pinning is imperfectly known \cite{klik+Chang_2013_thermal,puppin+Zani_2004_magnetic}. Our research may contribute to a clearer understanding of the said effects.
	
	 In many cases, the work of Flandoli and Gatarek \cite{Flandoli_Gatarek} is used to establish the tightness of the family of laws of the minimizing sequence of solutions. But when jump noise is taken into account, this method may not be applicable.
	The approach by M\'etivier \cite{Metivier_SPDE_InfDimensions_Book}, Brze\'zniak, \textit{et. al.} \cite{ZB+EH_2009_MaximalRegularlty_StochasticConvolutions_LevyNoise,ZB+EH+Paul_2018_StochasticReactionDiffusion_JumpProcesses},  Motyl \cite{Motyl_2013_SNSE_LevyNoise} depends on a deterministic compactness result, which further depends on some energy bounds and the Aldous condition-a stochastic variant of the Arz\'ela and Ascoli's equicontinuity result. A similar idea has been used in \cite{ZB+UM_SLLGE_JumpNoise,ZB+UM_WeakSolutionSLLGE_JumpNoise} for the stochastic LLG equation. We use the latter approach here.

	We follow the works \cite{ZB+BG+TJ_Weak_3d_SLLGE,ZB+BG+Le_SLLBE} and add the noise to the effective field. Again, taking motivation from \cite{D+M+P+V}, see also \cite{ZB+UM+SG_2022Preprint_SLLGE_Control,UM+SG_2022Preprint_SLLBE_RelaxedControl}, we extend the effective field to include the external control. Here, we add an operator $L$ that can depend (potentially nonlinearly) on both the control process $u$ and the corresponding solution process $m$. Rather than simply adding the external control, we do what is described above.
	
		We first state the problem that we have considered. Let $\mathcal{O}\subset \mathbb{R}^d,d=1,2,3$ be a bounded domain with a smooth boundary $\partial \mathcal{O}$. 
  Let $B$ denote the unit ball in $\mathbb{R}^{N}$ (excluding the center) for a fixed $N\in\mathbb{N}$. Let $T>0$ denote the finite time horizon. Let $\mathbb{U}$ denote the control space that we consider. We assume that $\mathbb{U}$ is a metrizable Suslin space. Consider operator $L: L^2 \times \mathbb{U} \to L^2$. This is the control operator. We consider the following controlled LLB equation with pure jump noise in the Marcus canonical sense. Note that we have replaced the constants $\kappa_1,\kappa,\gamma,\mu$ by $1$ for simplicity.
	\begin{align}\label{eqn controlled problem considered with L}
		\nonumber m(t) = \, & m_0 + \int_{0}^{t} \l[ \Delta m(s) + m(s) \times \Delta m(s) - \l( 1 + \l| m(s) \r|_{\mathbb{R}^3}^2 \r) m(s) + L\l(m(s),u\r) \r] \, ds \\
		& + \int_{0}^{t}\int_{B} \l( m(s) \times h + h \r) \diamond d\mathbb{L}(s),\ t\in[0,T].
	\end{align}
	with the Neumann boundary condition. Here $\mathbb{L}(t) = \l(L_1(t), \dots, L_N(t)\r), N\in\mathbb{N}$ is a $\mathbb{R}^N$-valued L\'evy process with pure jump. For the sake of simplicity, we assume $N = 1$.
 The meaning of the stochastic integral in \eqref{eqn controlled problem considered with L} and its interpretation is explained briefly in Section \ref{section Marcus form}.
	$\mathbb{L}(t)$ is a $\mathbb{R}$-valued L\'evy process with pure jump (i.e., the continuous part of the process is 0). The L\'evy process $\mathbb{L}$ is understood as follows.
	\begin{equation}\label{eqn meaning of L eqn 1}
		L(t) = \int_{0}^{t} \int_{B} l \, \tilde{\eta}(ds,dl) + \int_{0}^{t} \int_{B^c} l \, \eta(ds,dl),\ t\geq 0.
	\end{equation}
	Here $\eta$ denotes a time homogeneous Poisson random measure and $\tilde{\eta}$ denotes the corresponding compensated time homogeneous Poisson random measure, with a compensator $\text{Leb} \, \otimes \nu$ $\text{ i.e. }\l( \tilde{\eta} = \eta - \text{Leb} \otimes \nu \r)$.\\
	\textbf{Associated Cost:}\\
 We now state the cost $J$ that is associated with the control problem \eqref{eqn controlled problem considered with L}. For a control process $u$ and a corresponding solution process $m$, the cost $J$ is given by
	\begin{equation}\label{eqn cost functional}
		J(m,u) = \mathbb{E} \int_{0}^{T} F(m(t),u) \, dt.
	\end{equation}
	Here, $F$ denotes the running cost and $\mathbb{E}$ denotes the expectation with respect to the given probability space.
	
	\section{Some Definitions, Statement of the main result.}
 Let us define an operator $A$ on its domain $D(A)\subset L^2$ to $L^2$ as follows. 
	\begin{align}\label{Definition of Neumann Laplacian}
		\begin{cases}
			D(A) &:= \l\{v\in H^2 : \frac{\partial v}{\partial \vartheta} = 0\ \text{on}\ \partial \mathcal{O} \r\}, \\
			Av &:= -\Delta v, \ \text{for}\ v\in D(A).
		\end{cases}	
	\end{align}
 Here, $\Delta$ denotes the Neumann Laplacian operator and $\vartheta$ denotes the outward pointing normal vector to the boundary $\partial \mathcal{O}$.
 For $\beta \geq 0$, let us define the space	\begin{equation}\label{definition of the space X beta}
		X^{\beta} = \text{dom}\l(A_1^{\beta}\r).
	\end{equation}
Its dual space is denoted by $X^{-\beta}$.
	For $\beta = 0$, we have
	$$X^0 = L^2.$$
	\subsection{The Marcus mapping, Marcus canonical form}\label{section Marcus form}
	This subsection describes the Marcus mapping $\Phi$ and the resulting form (Marcus canonical form) of the equation \eqref{eqn controlled problem considered with L}.
	
	Let $h\in W^{2,\infty}$ denote the given data. Let us define a mapping $g$ as follows.
	\begin{equation}
		\l\{ g:H^1 \ni v \mapsto v \times h \in H^1 \r\}.
	\end{equation}
	The map $g$ is a bounded linear map.	
	We now define a map 
	\begin{equation}
		\Phi: \mathbb{R}_{+} \times \mathbb{R} \times H^1 \to H^1,
	\end{equation}
	such that for every $x\in H^1$, the mapping
	\begin{equation}
		t\mapsto \Phi(t,l,x)
	\end{equation}
	is a $C^1$ solution to the ordinary differential equation
	\begin{equation}
		\frac{d\Phi}{dt}(t,l,x) = l g\l(\Phi(t,l,x)\r),
	\end{equation}
	with initial condition $\Phi(0,l,x) = x$, $x\in H^1$.
	
	The operator $g$ is bounded linear. This implies that the operator $\Phi$ is well defined. Arguing similarly, one can show that the operator
	\begin{equation*}
		\l\{ \mathbb{R}_{+}\times \mathbb{R} \times L^2 \ni \l( t,l,x \r) \mapsto \Phi\l( t,l,x \r) \in L^2 \r\},
	\end{equation*}
	is also well defined. In what follows, let us fix $t = 1$. $\Phi(l,\cdot) := \Phi(1,l,\cdot)$, $l\in\mathbb{R}$.\\
	\textbf{Some properties of the Marcus Mapping $\Phi$:}
 We now state some properties of the mapping $\Phi$. For proofs, we refer the reader to \cite{ZB+UM_WeakSolutionSLLGE_JumpNoise}. The reader can also refer to the preprints \cite{ZB+UM+Zhai_Preprint_LDP_LLGE_JumpNoise,UM+SG_2022Preprint_SLLBE_LDP}.
	\begin{lemma}\label{lemma linear growth Phi}
		Let $X \in \l\{ L^2, H^1 \r\}$. Then there exists a constant $C>0$ such that for every $l\in B$, the following hold.

		\begin{equation}\label{eqn linear growth for Phi}
			\l| \Phi(l,x) \r|_{X} \leq C  \l( 1 + \l| x \r|_{X} \r),\ x\in X.
		\end{equation}
		
		\begin{equation}\label{eqn linear growth 2 for Phi}
			\l| \Phi(l,x) \r|_{X}^2 \leq C  \l( 1 + \l| x \r|_{X}^2 \r),\ x\in X.
		\end{equation}
		
	\end{lemma}

	Consider the following mappings on $L^2$.
	\begin{align}\label{eqn definition of G}
		G(l,v) = \Phi(l,v) - v,\ v\in L^2:
	\end{align}
	\begin{align}\label{eqn definition of H}
		H(l,v) = \Phi(v) - v - l g(v),\ v\in L^2.
	\end{align}
	With the above two mappings in mind, we state the following lemma.
	
	\begin{lemma}\label{lemma linear growth Lipschitz G and H}
		There exists a constant $C>0$ such that, for every $u,v\in L^2$ and $l\in B$, the following inequalities hold.
		\begin{align}\label{eqn G linear growth}
			\l|G(l,v)\r|_{L^2} \leq C \l( 1 + \l| v \r|_{L^2}\r),
		\end{align}
		\begin{align}\label{eqn G Lipschitz}
			\l| G(l,u) - G(l,v) \r|_{L^2} \leq C \l| u - v \r|_{L^2},
		\end{align}
		\begin{align}\label{eqn H linear growth}
			\l|H(l,v)\r|_{L^2} \leq C \l( 1 + \l| v \r|_{L^2}\r),
		\end{align}
		\begin{align}\label{eqn H Lipschitz}
			\l| H(l,u) - H(l,v) \r|_{L^2} \leq C \l| u - v \r|_{L^2}.
		\end{align}
	\end{lemma}
 We now introduce another operator $b$ defined by
 \begin{equation}\label{eqn definition of b}
     b(v) = \int_{B} H(l,v) \, \nu(dl),\ v\in L^2.
 \end{equation}
	\noindent	
	\textbf{Equation in Marcus Canonical Form:}
 So far, we have described the Marcus mapping $\Phi$, along with some of its properties. We now use this, along with the operators $G,H,b$ defined in \eqref{eqn definition of G}, \eqref{eqn definition of H}, \eqref{eqn definition of b} above to understand equation \eqref{eqn controlled problem considered with L}. More specifically, the equation \eqref{eqn controlled problem considered with L} can be understood as the following.	
	\begin{align}\label{eqn controlled problem considered Marcus Form}
		\nonumber m(t) = & \, m_0 + \int_{0}^{t} \l[ \Delta m(s) + m(s) \times \Delta m(s) - \l( 1 + \l| m(s) \r|_{\mathbb{R}^3}^2 \r) m(s) + L\l(m(s),u\r) \r] \, ds \\
		& + \int_{0}^{t} b\bigl(m(s)\bigr) \, ds + \int_{0}^{t} \int_{B} G\bigl(l,m(s)\bigr) \, \tilde{\eta}(dl,ds),\ t\in[0,T].
	\end{align}
	
	\subsection{Relaxed Controlled Equation}
 In this subsection, we use the theory of Young measures to reduce the problem \eqref{eqn controlled problem considered Marcus Form} to a relaxed form, for which we show the existence of an optimal control. For a random Young measure $\lambda$, the relaxed controlled equation corresponding to the equation \eqref{eqn controlled problem considered Marcus Form} is given by
	\begin{align}\label{eqn relaxed controlled problem considered Marcus Form}
		\nonumber m(t) = & \, m_0 + \int_{0}^{t} \l[ \Delta m(s) + m(s) \times \Delta m(s) - \l( 1 + \l| m(s) \r|_{\mathbb{R}^3}^2 \r) m(s) +  \r] \, ds \\
		\nonumber & + \int_{0}^{t} \int_{\mathbb{U}} L\l(m(s),v\r) \lambda(dv,ds) \\
		& + \int_{0}^{t} b\bigl(m(s)\bigr) \, ds + \int_{0}^{t} \int_{B} G\bigl(l,m(s)\bigr) \, \tilde{\eta}(dl,ds),\ t\in[0,T].
	\end{align}
	\textbf{Associated Relaxed Cost:}\\
	For a weak martingale solution $\pi$ (see Definition \ref{definition weak martingale solution}), the relaxed cost $\mathscr{J}(\pi)$ (corresponding to \eqref{eqn relaxed controlled problem considered Marcus Form}) is given by
	\begin{equation}\label{eqn relaxed cost functional}
		\mathscr{J}(\pi) = \mathbb{E} \int_{0}^{T} \int_{\mathbb{U}} F(m(t),v) \, \lambda(dv,dt),
	\end{equation}
 with $F,\mathbb{E}$ being the same as in \eqref{eqn cost functional}.
 	
	\subsection{Definitions of some spaces}\label{section some spaces}
	Later on, we use certain compactness results from \cite{Aldous_1978_Stopping_times_and_tightness,Metivier_SPDE_InfDimensions_Book,Motyl_2013_SNSE_LevyNoise}, among others.
	Towards that, we define some spaces which will be used. We have borrowed definitions, results from Billingsley \cite{Billingsley_1999_ConvergenceOfProbabilityMeasures_Book}, Parthasarthy \cite{Parthasarathy_1967_Book_ProbabilityMeasuresOnMetricSpaces_Book}, M\'etivier \cite{Metivier_SPDE_InfDimensions_Book}, Aldous \cite{Aldous_1978_Stopping_times_and_tightness,Aldous_1989_StoppingTimes2} among others.

	\begin{enumerate}
		\item The space $\mathbb{D}([0,T] : X^{-\beta})$ : This is the space of all c\`adl\`ag functions $v: [0,T] \to X^{-\beta}$ with the topology induced by the Skorohod metric $\delta_{T,X^{-\beta}}$ given by
		
		\begin{align}
			\nonumber \delta_{T,X^{-\beta}}(x,y) : = \inf_{\lambda\in\Lambda_T} \bigg[ & \sup_{t\in[0,T]}  \rho\bigl( x(t) , y( \lambda(t) ) \bigr) + \sup_{t\in[0,T]} \l| t - \lambda(t) \r| \\
			& + \sup_{s<t}\l| \log \l( \frac{ \lambda(t) - \lambda(s) }{ t - s } \r) \r| \bigg].
		\end{align}
		Here $\Lambda_T$ is the set of all increasing homeomorphisms $[0,T]$ and $\rho$ denotes the norm generated metric on the space $X^{-\beta}$.

		\item The space $L^2_{\text{w}}(0,T:H^2)$ : This is the space $L^2(0,T:H^2)$ endowed with the weak topology.
		\item The space $\mathbb{D}([0,T] : H^1_{\text{w}})$ : This is the space of all weakly c\'adl\'ag functions $v: [0,T] \to H^1$, with the weakest topology such that for all $u\in H^1$, the maps
		\begin{align*}
			\mathbb{D}([0,T] : H^1_{\text{w}}) \ni v \mapsto \l\langle v (\cdot) , u \r\rangle_{H^1} \in \mathbb{D}([0,T] : \mathbb{R}),
		\end{align*}
		are continuous.
	\end{enumerate}
	\subsection{Definitions, statement of the main result}
	\begin{definition}\label{definition weak martingale solution}
		A weak martingale solution to the problem \eqref{eqn relaxed controlled problem considered Marcus Form} is a tuple
		\begin{align*}
			\l( \Omega^\p , \mathcal{F}^\p , \mathbb{F}^\p , \mathbb{P}^\p , M^\p , \eta^\p, \lambda^\p\r),
		\end{align*}
		such that
		\begin{enumerate}
			\item $\l( \Omega^\p , \mathcal{F}^\p , \mathbb{F}^\p , \mathbb{P}^\p\r)$ is a probability space (with filtration $\mathbb{F}^\p = \l\{ \mathcal{F}_t^\p \r\}_{t\in[0,T]}$) satisfying the usual hypotheses.
			\item $\eta^\p$ is a time homogeneous Poisson random measure on the measurable space $\l(B,\mathcal{B}\r)$ over $\l( \Omega^\p , \mathcal{F}^\p , \mathbb{F}^\p , \mathbb{P}^\p\r)$ with intensity measure $Leb \otimes \nu$.
			\item $\lambda^\p$ is a random Young measure on the space $\mathbb{U}$ having the same law as $\lambda$ on the space $\mathcal{Y}(0,T:\mathbb{U})$.
			\item $\m:[0,T] \times \Omega ^\p \to H^1$ is an $\mathbb{F}^\p$-progressively measurable weakly c\'adl\'ag process which satisfies the following estimates. There exists a constant $C>0$ such that
			\begin{enumerate}
				\item \begin{equation}\label{eqn weak martingale solution definition bound 1}
					\mathbb{E}^\p\sup_{t\in[0,T]}\l|\m(t)\r|_{H^1}^2 \leq C,
				\end{equation}
				\item \begin{equation}\label{eqn weak martingale solution definition bound 2}
					\mathbb{E}^\p\int_{0}^{T} \l|\m(t)\r|_{H^2}^2 \, dt \leq C.
				\end{equation}
			\end{enumerate}
			\item For each $V\in L^4(\Omega^\p: H^{1})$, \eqref{eqn relaxed controlled problem considered Marcus Form} is satisfied $\mathbb{P}^\p$-a.s. in the weak (PDE) sense as follows.
			\begin{align}\label{eqn weak martingale solution definition eqn weak form}
				\nonumber \m(t) = & \, m_0 - \int_{0}^{t}  \l\langle \nabla \m(s) , \nabla V \r\rangle_{L^2} \, ds +  \int_{0}^{t}  \l\langle \m(s) \times \nabla \m(s) , \nabla V \r\rangle_{L^2} \, ds \\
				\nonumber & -  \int_{0}^{t} \int_{\mathcal{O}} \l( 1 + \l|\m(s,x)\r|_{\mathbb{R}^3}^2 \r) \l\langle \m(s,x) ,  V(x) \r\rangle_{\mathbb{R}^3} \, dx \, ds \\
				\nonumber & + \int_{0}^{t} \int_{\mathbb{U}} \l\langle L\l(\m(s),v\r) , V \r\rangle_{L^2} \, \lambda^\p(dv,ds)\\
				& + \int_{0}^{t} \l\langle b\bigl(l,\m(s)\bigr) , V \r\rangle_{L^2} \, ds + \int_{0}^{t} \int_{B} \l\langle G\bigl(l,\m(s)\bigr) , V \r\rangle_{L^2} \, \tilde{\eta}(dl,ds),\ t\in[0,T].
			\end{align}
		\end{enumerate}
	\end{definition}
	
	\begin{definition}\label{definition optimal control}
		A weak martingale solution $\pi$ is said to be admissible if $\mathscr{J}(\pi) < \infty$. We denote by $\mathcal{U}_{m_0,T}$, the class of all admissible solutions with initial data $m_0$ and terminal time $T$. A weak martingale solution $\pi^* \in \mathcal{U}_{m_0,T}$ is said  to be an optimal control to the relaxed problem \eqref{eqn relaxed controlled problem considered Marcus Form} with the associated cost \eqref{eqn relaxed cost functional} if $\pi^*$ minimizes the cost. That is, if
		\begin{equation}
	         \mathscr{J}(\pi^*) = \inf_{\pi\in\mathcal{U}_{m_0,T}} \mathscr{J}(\pi) = \Lambda \text{ (say) }.
		\end{equation}
		
		\begin{assumption}\label{assumption}
			Assumptions
			\begin{enumerate}
				\item We assume the following on the control operator $L$.
				\begin{enumerate}
					\item There exists an inf-compact function $\kappa$ and a constant $C$ such that for some $0 \leq r < 2$, the following inequality holds.
					\begin{equation}\label{eqn assumption on the operator L 3 terms}
						\l| L(m,v) \r|_{L^2} \leq C \bigl( \l| m \r|_{L^2}^r + \kappa(\cdot,v) + \l| m \r|_{L^2}^r \kappa(\cdot,v) \bigr).
					\end{equation}
					In what follows, we consider only the third term on the right hand side of the above inequality. This does not change the complexity of the problem, while reducing certain calculation lengths. Therefore the assumption \eqref{eqn assumption on the operator L 3 terms} reduces to
					\begin{equation}\label{eqn assumption on the operator L}
						\l| L(m,v) \r|_{L^2} \leq C \l| m \r|_{L^2}^r \kappa(\cdot,v).
					\end{equation}
     
					\item We assume that the function $L$ is continuous in the following sense. For any $R>0$, let $K_R = \l\{ v\in\mathbb{U} : \kappa(\cdot,v) \leq R \r\}$. Let $m_n \to m$ in $L^4(\Omega:L^4(0,T:L^4))$ as $n\to\infty$. Then
					\begin{align}
						\mathbb{E}\int_{0}^{T} \sup_{v\in K_R} \l| L(m_n,v) - L(m,v) \r|_{L^2}^2 \, ds \to 0,\ \text{as} \ n\to\infty.
					\end{align}
					
					\item For each $m\in L^2$, the mapping
					\begin{equation}
						L(m,\cdot):\mathbb{U} \to L^2,
					\end{equation}
					is continuous.
				\end{enumerate}
				\item We assume the following for the inf-compact function $\kappa$.
    \begin{enumerate}
    \item For the function $\kappa$ mentioned above, there exists a constant $C>0$ such that
				\begin{equation}\label{eqn coercivity assumption}
					F(t,m(t),v) \geq C \kappa^{4} (t,v),\ \forall v\in \mathbb{U},\ t\in[0,T],\ \forall m\in H^1.
				\end{equation}
    \item For the inf-compact function $\kappa$, we assume that for any random Young measure $\lambda$,
     \begin{equation}\label{eqn assumption on kappa}
         \mathbb{E} \int_{0}^{T} \kappa^4(s,v) \, \lambda(dv,ds) < \infty.
     \end{equation}
     
     \end{enumerate}
				\item The cost functional $F$ is lower semicontinuous on the space $(H^1)^\p$. That is, if $m_n\to m$  as $n\to\infty$ in $(H^1)^\p$ then
				\begin{equation}
					F(m) \leq \liminf_{n\to\infty} F(m_n).
				\end{equation}
				\item For the given data $h$ and initial data $m_0$, we assume the following.
				\begin{equation}
					h\in W^{2,\infty} \text{ and } m_0\in H^1.
				\end{equation}
			\end{enumerate}
		\end{assumption}
	\end{definition}
	
	\begin{theorem}\label{thm Optimal control}
		Let $\mathcal{O}\subset \mathbb{R}^d,d=1,2,3$ denote a bounded domain with smooth boundary. Let the terminal time $0<T<\infty$ be fixed. Further, let the Assumption \ref{assumption} hold. Then the relaxed control problem \eqref{eqn relaxed controlled problem considered Marcus Form} admits an optimal control with the cost \eqref{eqn relaxed cost functional}. In other words, the problem \eqref{eqn controlled problem considered with L} admits a relaxed optimal control with cost \eqref{eqn cost functional}.
	\end{theorem}
	
	\section{Proof of Theorem \ref{thm Optimal control}}
	\begin{proof}[Heuristic idea of the proof]
		First of all, we can observe that if we let $L=0$, the problem reduces to an uncontrolled problem, for which one can perhaps show the existence of a weak martingale solution (see for instance the preprint \cite{UM+SG_2022Preprint_SLLBE_LDP}, see also \cite{UM+SG_2022Preprint_SLLBE_RelaxedControl}). 
		Moreover, we have assumed a fairly general cost functional, without any special convexity assumptions. Hence it makes sense for us to assume that there exists at least one weak martingale solution with finite cost. Otherwise, if there is no solution with finite cost, then the problem reduces to showing the existence of a weak martingale solution. We therefore assume that the infimum ($\Lambda$) of the cost is finite. That is,
		\begin{equation}
			\inf_{\pi\in\mathcal{U}_{m_0,T}}\mathscr{J}(\pi) = \Lambda < \infty.
		\end{equation}
		Therefore, there exists a minimizing sequence 
  $$\l\{ \pi_n \r\}_{n\in\mathbb{N}} = \l( \Omega_n , \mathbb{F}_n,\mathcal{F}_n,\mathbb{P}_n,m_n,\lambda_n,W_n\r),$$
  of weak martingale solutions. In particular,
		\begin{equation}\label{eqn pi n is a minimizing sequence}
			\lim_{n\to\infty} \mathscr{J}(\pi_n) = \Lambda.
		\end{equation}
		Therefore for this sequence, the corresponding cost is uniformly bounded. Using this information, we establish some uniform bounds for the solution processes $m_n,n\in\mathbb{N}$ in the weak martingale solution tuples $\pi_n$. Using these bounds, we obtain a subsequence of the solution processes which converges (in some sense) to another weak martingale solution of the same problem. We then show that the obtained solution indeed minimizes the cost, thus concluding the proof of the theorem.
	\end{proof}
	The remainder of this section is divided into multiple subsections and lemmas, which culminate into the proof of Theorem \ref{thm Optimal control}.
	\subsection{Uniform Energy Estimates}
 We now obtain uniform bounds on the processes $m_n,n\in\mathbb{N}$. The proofs are applications of the It\^o formula to the functions
 \begin{equation}
     \l\{ H^1 \ni v \mapsto \frac{1}{2} \l| v \r|_{L^2}^2 \in \mathbb{R} \r\} ,
 \end{equation}
 and
 \begin{equation}
     \l\{ H^1 \ni v\mapsto \frac{1}{2} \l| \nabla v \r|_{L^2}^2 \in \mathbb{R} \r\} ,
 \end{equation}
 to the processes $m_n,n\in\mathbb{N}$, followed by simplifications and applying the Gronwall inequality. One has to be careful as the processes $m_n,n\in\mathbb{N}$ are infinite dimensional. One can apply the It\^o formula from the work of Gyongy and Krylov \cite{Gyongy+Krylov_1982_StochasticEquationsSemimartingales}, see also Appendix in \cite{ZB+UM_WeakSolutionSLLGE_JumpNoise}. We skip the proofs here for brevity, and refer the reader to \cite{ZB+BG+Le_SLLBE,ZB+UM_WeakSolutionSLLGE_JumpNoise} for similar arguments. One can also look at the preprints \cite{UM+SG_2022Preprint_SLLBE_LDP,UM+SG_2022Preprint_SLLBE_RelaxedControl}. for some related results in detail.
	\begin{lemma}\label{lemma bounds 1}
		There exists a constant $C>0$ such that for every $n\in\mathbb{N}$, the following hold.
		\begin{equation}\label{eqn L infinity L2 bound mn}
			\mathbb{E}^n \sup_{t\in[0,T]} \l| m_n(t) \r|_{L^2}^2 \leq C,
		\end{equation}
		\begin{equation}\label{eqn L 2 H1 bound mn}
			\mathbb{E}^n \int_{0}^{T} \l| m_n(t) \r|_{H^1}^2 \, dt \leq C,
		\end{equation}
		\begin{equation}\label{eqn L 4 L4 bound mn}
			\mathbb{E}\int_{0}^{T} \l| m_n(t) \r|_{L^4}^4 \, dt \leq C.
		\end{equation}
Note that $E^n$ denotes the expectation on the space $\l( \Omega_n , \mathbb{F}_n,\mathcal{F}_n,\mathbb{P}_n \r)$, for $n\in\mathbb{N}$.
	\end{lemma}
	
	\begin{lemma}\label{lemma bounds 2}
		There exists a constant $C>0$ such that for every $n\in\mathbb{N}$, the following hold.
		\begin{equation}\label{eqn L infinity H1 bound mn}
			\mathbb{E}^n \sup_{t\in[0,T]} \l| m_n(t) \r|_{H^1}^2 \leq C,
		\end{equation}

		\begin{equation}\label{eqn L 2 H2 bound mn}
			\mathbb{E}^n \int_{0}^{T} \l| \Delta m_n(t) \r|_{L^2}^2 \, dt \leq C.
		\end{equation}
	\end{lemma}

	\subsection{Tightness of laws and subsequent convergence results}

 	Let us fix a notation. For $\beta\geq 0$, $p\geq 1$ and $1\leq q \leq 6$, let
	\begin{equation}\label{eqn definition of the space ZT}
		\mathcal{Z}_T = \mathbb{D}([0,T]:X^{-\beta})\cap L^p(0,T:L^q)\cap\mathbb{D}([0,T]:H^1_{\text{w}})\cap L^2_{\text{w}}(0,T:H^2).
	\end{equation}
	
	We endow $\mathcal{Z}_T$ with the smallest topology under which the inclusion maps (for the intersection into all the spaces) are continuous. For the definitions of the above spaces, see Section \ref{section some spaces}.

 We state a few definitions and results, particular cases of which will be used for obtaining tightness for the laws of $m_n$ on the space $\mathcal{Z}_T$, along with Lemma \ref{Lemma general tightness 3} which is borrowed from \cite{ZB+UM_WeakSolutionSLLGE_JumpNoise}. For more details, we refer the reader to Aldous \cite{Aldous_1978_Stopping_times_and_tightness,Aldous_1989_StoppingTimes2}, Motyl \cite{Motyl_2013_SNSE_LevyNoise}, M\'etivier \cite{Metivier_SPDE_InfDimensions_Book}, Parthasarathy \cite{Parthasarathy_1967_Book_ProbabilityMeasuresOnMetricSpaces_Book}.	
	We now state the Aldous condition for tightness from \cite{Metivier_SPDE_InfDimensions_Book}.
	\begin{definition}[Aldous Condition]\label{definition Aldous Condition}
		Let $\{ X_n \}_{n\in\mathbb{N}}$ be a sequence of c\'adl\'ag, adapted stochastic processes in a Banach space $E$. We say that the sequence $\{ X_n \}_{n\in\mathbb{N}}$ satisfies the Aldous condition on $E$ if for every $\varepsilon>0$ and $\eta>0$ there is $\delta>0$ such that for every sequence $\l\{ \tau_n\r\}_{n\in\mathbb{N}}$ of $\mathbb{F}$-stopping times with $\tau_n + \theta \leq T$ one has
		\begin{equation}
			\sup_{n\in\mathbb{N}} \sup_{0\leq \theta \leq \delta} \mathbb{P} \bigl\{ \| X_n\l( \tau_n + \theta \r) - X_n\l( \tau_n \r) \|_{E} \geq \eta \bigr\} \leq \varepsilon.
		\end{equation}
	\end{definition}	
	We now state a result (see M\'etivier \cite{Metivier_SPDE_InfDimensions_Book}, Motyl \cite{Motyl_2013_SNSE_LevyNoise}), a condition which guarantees that the sequence $\{ X_n \}_{n\in\mathbb{N}}$ satisfies the Aldous condition (as in Definition \ref{definition Aldous Condition}) in a separable Banach space.
	\begin{lemma}\label{Lemma alternate definition Aldous Condition}
		Let $E$ be a separable Banach space. Let $\{ X_n \}_{n\in\mathbb{N}}$ be a sequence of $E$-valued random variables. Assume that for every sequence $\l\{ \tau_n\r\}_{n\in\mathbb{N}}$ of $\mathbb{F}$-stopping times with $\tau_n \leq T$, and $\theta \geq 0$, with $\tau_n+\theta \leq T$, we have
		\begin{equation}
			\mathbb{E}\biggl[ \bigl\| X_n(\tau_n + \theta) - X_n(\tau_n) \bigr\|_{E}^{\gamma_1} \biggr] \leq C \theta^{\gamma_2},
		\end{equation}
		for some $\gamma_1,\gamma_2>0$ and some constant $C>0$. Then the sequence $\l\{ X_n \r\}_{n\in\mathbb{N}}$ satisfies the Aldous condition in $E$.
	\end{lemma}

  We again skip the proofs of the following result and refer the reader to \cite{ZB+UM_SLLGE_JumpNoise}, for example, for similar results. The reader can also refer to the preprints \cite{ZB+UM+Zhai_Preprint_LDP_LLGE_JumpNoise,UM+SG_2022Preprint_SLLBE_LDP}.
	
	\begin{lemma}\label{lemma Aldous Condition for mn}
		The sequence $\{ m_n \}_{n\in\mathbb{N}}$ satisfies the Aldous condition on the space $X^{-\beta}$ for $\beta > \frac{1}{4}$.
	\end{lemma}

For $p,q\geq 1$ and $\beta\geq 0$, let us define
	\begin{equation}
		Z_T : = \mathbb{D}([0,T]:X^{-\beta})\cap L^p(0,T:L^q)\cap\mathbb{D}([0,T]:H^1_{\text{w}}).
	\end{equation}
	The space is equipped with the Borel $\sigma$-algebra (generated by the open sets in the locally convex topology of $Z_T$).
 For proof of the following Lemma \ref{Lemma general tightness 3}, we refer the reader to Proposition 5.11, \cite{ZB+UM_WeakSolutionSLLGE_JumpNoise}.
 \begin{lemma}\label{Lemma general tightness 3}
		Let $\l\{X_n\r\}_{n\in\mathbb{N}}$ be a sequence of c\'adl\'ag adapted $X^{-\beta}$-valued processes satisfying
		\begin{enumerate}
			\item[(a)] 
			\begin{equation}
				\sup_{n\in\mathbb{N}} \mathbb{E}\l[ 
				\l| X_n \r|_{L^{\infty}(0,T:H^1)}^2 \r] < \infty,
			\end{equation}
			\item[(b)] the Aldous condition in $X^{-\beta}$.
		\end{enumerate}
		Then the sequence $\l\{\mathbb{P}^{X_n}\r\}_{n\in\mathbb{N}}$ is tight in $Z_T$. That is for given $\varepsilon>0$, there exists a compact set $K_{\varepsilon} \subset Z_T$ with
		\begin{equation}
			\mathbb{P}^{X_n}(K_{\varepsilon}) \geq 1 - \varepsilon,
		\end{equation}
		for all $n\in\mathbb{N}$.
	\end{lemma}	
	Using the bounds established in Lemma \ref{lemma bounds 1}, Lemma \ref{lemma bounds 2} along with Lemma \ref{Lemma general tightness 3} and the compact embedding of the space $L^2(0,T:H^2)$ into the space $L_{\text{w}}^2(0,T:H^2)$, we have the following lemma.
 \begin{lemma}\label{lemma tightness lemma}
		The sequence of laws $\l\{\mathcal{L}(m_n)\r\}_{n\in\mathbb{N}}$ is tight on the space $\mathcal{Z}_T$.
	\end{lemma}

	In order to talk about the convergence (of laws) for the pair $\l( m_n , \eta \r)$, we use a generalization of the Skorohod Theorem (see \cite{ZB+EH+Paul_2018_StochasticReactionDiffusion_JumpProcesses}). We use the said result to obtain another sequence of $\mathcal{Z}_T$-valued random variables (possibly on a different probability space). Let $\eta_n := \eta,n\in\mathbb{N}$. With this notation and Lemma \ref{lemma tightness lemma} and the Skorohod Theorem, we can conclude that the sequence $\l\{ \mathcal{L}\l(m_n,\eta_n\r)\r\}_{n\in\mathbb{N}}$ is tight on the space $\mathcal{Z}_T \times \mathbb{M}_{\bar{\mathbb{N}}}\l([0,T] \times B\r)$, where $\mathbb{M}_{\bar{\mathbb{N}}}(\mathcal{S})$ denotes all the $\mathbb{N}\cup\{\infty\}$-valued measures on the metric space $\l(\mathcal{S},\mathscr{\varrho}\r)$ (see for instance \cite{ZB+UM_WeakSolutionSLLGE_JumpNoise}).
	By applying a generalized Jakubowski-Skorohod Representation Theorem \cite{ZB+EH+Paul_2018_StochasticReactionDiffusion_JumpProcesses}, we have the following lemma.
 
 Using Theorem 2.13, \cite{ZB+RS}, the coercivity assumption \eqref{eqn coercivity assumption}, along with the uniform bounds on the cost, we conclude that the sequence of laws of $\l\{\lambda_n\r\}_{n\in\mathbb{N}}$ is tight on the space $\mathcal{Y}(0,T:\mathbb{U})$.

	\begin{lemma}\label{lemma use of Skorohod Theorem}
		There exists a probability space $$\l( \Omega^\p , \mathbb{F}^\p , \mathcal{F}^\p , \mathbb{P}^\p\r)$$ and another sequence of random variables on that space $\{ \m_n \}_{n\in\mathbb{N}}$, along with $\mathcal{Z}_T \times \mathbb{M}_{\bar{\mathbb{N}}}\l([0,T] \times B\r) \times \mathcal{Y}(0,T:\mathbb{U})$-valued random variables $\l(\m,\eta^\p,\lambda^\p\r) , \l(\m_n,\eta_n^\p,\lambda_n^\p\r),n\in\mathbb{N}$ on the aforementioned probability space, such that the following hold.
		\begin{enumerate}
			\item \begin{equation}
				\mathcal{L}\l( m_n,\eta_n,\lambda_n \r) = \mathcal{L}\l( \m_n,\eta_n^\p,\lambda_n^\p \r), \forall n\in\mathbb{N},\ \text{on}\ \mathcal{Z}_T \times \mathbb{M}_{\bar{\mathbb{N}}}\l([0,T] \times B\r) \times \mathcal{Y}(0,T:\mathbb{U}),
			\end{equation}
			\item \begin{equation}
				\l( \m_n,\eta_n^\p, \lambda_n^\p \r) \to \l( \m,\eta^\p, \lambda^\p  \r),\ \text{as}\ n\to\infty, \ \mathbb{P}^\p -a.s. \ \text{in}\ \mathcal{Z}_T \times \mathbb{M}_{\bar{\mathbb{N}}}\l([0,T] \times B\r) \times \mathcal{Y}(0,T:\mathbb{U}),
			\end{equation}
			\item \begin{equation}
				\eta_n^\p(\omega^\p) = \eta^\p(\omega^\p),\ \text{for all} \  \omega^\p\in\Omega^\p.
			\end{equation}
			
		\end{enumerate}
	\end{lemma}
	
	For the remainder of the section, we fix $p = q = 4$ and $\beta = \frac{1}{2}$ for the space $\mathcal{Z}_T$. Note that $X^{-\frac{1}{2}}$ can be identified with the space $(H^1)^\p$, which is the dual of the space $H^1$.
	
	By appealing to Kuratowski Theorem , see Theorem 1.1 in \cite{Vakhania_Probability_distributions_on_Banach_spaces}, (see also \cite{ZB+Dhariwal_2012_SNSE_LevyNoise}), we conclude that the processes $\m_n$ and $m_n$ satisfy the same bounds. This is made precise in the following lemma.
	
	\begin{lemma}\label{lemma bounds 1 lemma mn prime}
		Let $p\geq 1$. Then, there exists a constant $C>0$, which can depend on $p$ but not on $n\in\mathbb{N}$ such that the following hold.
		\begin{equation}\label{eqn L infinity H1 bound mn prime}
			\mathbb{E}^\p \sup_{t\in[0,T]}\l| \m_n(t) \r|_{H^1}^{2p} \leq C,
		\end{equation}
		\begin{equation}\label{eqn L2 H2 bound mn prime}
			\mathbb{E}^\p \l( \int_{0}^{T} \l| \m_n(t) \r|_{H^2}^2 \, dt \r)^{p} \leq C.
		\end{equation}
            \begin{equation}\label{eqn L 4 L4 bound mn prime}
			\mathbb{E}^\p \int_{0}^{T} \l| \m_n(t) \r|_{L^4}^4 \, dt \leq C.
		\end{equation}
		Here $\mathbb{E}^\p$ denotes the expectation with respect to the probability space $\l(\Omega^\p,\mathbb{P}^\p\r)$.
	\end{lemma}
	Using the above bounds along with certain semicontinuity arguments, we can show the following bounds, in particular for $p=1$, for the limit process $\m$.
	\begin{lemma}\label{Lemma FG approximations bounds on m prime}
		There exists a constant $C>0$ such that the following hold.
		
		\begin{equation}\label{eqn L infinity H1 bound m prime}
			\mathbb{E}^\p \sup_{t\in[0,T]}\l| \m(t) \r|_{H^1}^{2} \leq C,
		\end{equation}
  
		\begin{equation}\label{eqn L2 H2 bound m prime}
			\mathbb{E}^\p \int_{0}^{T} \l| \m(t) \r|_{H^2}^2 \, dt  \leq C.
		\end{equation}
  \begin{equation}\label{eqn L 4 L4 bound m prime}
			\mathbb{E}^\p \int_{0}^{T} \l| \m(t) \r|_{L^4}^4 \, dt \leq C.
		\end{equation}
	\end{lemma}
	
	\begin{lemma}\label{lemma convergence lemma 1 of FG approximates}
		The following convergences hold.
		\begin{enumerate}
			\item \begin{equation}\label{eqn L4L4L4 convergence mn prime}
				\m_n\to\m\ \text{in}\ L^4\l( \Omega^\p : L^4(0,T:L^4)\r).
			\end{equation}
			\item 
			
			\begin{equation}
				\m_n \to  \m\ \text{weakly in}\ L^2\l( \Omega^\p : L^2(0,T : H^2 )\r).
			\end{equation}
			
			\item For $q\in (1,\frac{4}{3})$,
			\begin{equation}
				\m_n \times \Delta \m_n \to \m \times \Delta \m\ \text{weakly in}\ L^q(\Omega^\p:L^q(0,T:L^2)),
			\end{equation}
			and for $\beta>\frac{1}{4}$,
			\begin{equation}
				\m_n \times \Delta \m_n \to \m \times \Delta \m\ \text{weakly in}\  L^2(\Omega^\p:L^2(0,T:X^{-\beta})).
			\end{equation}

			\item  Let $V\in L^4(\Omega^\p:(0,T:H^1))$. Then
			\begin{equation}
				\lim_{n\to\infty}\mathbb{E}^\p \int_{0}^{T} \l\langle  \l( 1 + \l| \m_n(s) \r|_{\mathbb{R}^3}^2 \r) \m_n(s)  
				-\l( 1 + \l| \m(s) \r|_{\mathbb{R}^3}^2 \r) \m(s)  , V(s) \r\rangle_{L^2} \, ds = 0.
			\end{equation}
			\end{enumerate}
	\end{lemma}
	
	We skip the proof of this lemma and refer the reader to Section 5 in \cite{ZB+BG+Le_SLLBE} (see also \cite{ZB+BG+TJ_Weak_3d_SLLGE,ZB+BG+TJ_LargeDeviations_LLGE,UM+AAP_2021_LargeDeviationsSNSELevyNoise,ZB+Dhariwal_2012_SNSE_LevyNoise}). For a proof of the couple of lemmas that follow (Lemma \ref{lemma convergence for bn existence of weak martingale solution}, Lemma \ref{lemma convergence for Gn existence of weak martingale solution}), one can refer to the work \cite{ZB+UM_WeakSolutionSLLGE_JumpNoise}. In particular, one can refer to the preprint \cite{UM+SG_2022Preprint_SLLBE_LDP} for some detailed calculations.
	
	\begin{lemma}\label{lemma convergence for bn existence of weak martingale solution}
		Let $V\in L^4(\Omega^\p : L^2)$. Then the following convergence holds.
		\begin{equation}\label{eqn convergence for bn existence of weak martingale solution}
		\lim_{n\to\infty} \mathbb{E}^\p \l| \int_{0}^{T} \l\langle b_n\bigl(\m_n(s)\bigr) - b\bigl(\m(s)\bigr) , V \r\rangle_{L^2} \, ds \r|^2 = 0.
	\end{equation}
	\end{lemma}
	
	\begin{lemma}\label{lemma convergence for Gn existence of weak martingale solution}
		Let $V\in L^4(\Omega^\p:L^2)$. Then the following convergence holds.
		\begin{equation}\label{eqn convergence for Gn existence of weak martingale solution}
			\lim_{n\to\infty} \mathbb{E}^\p\l[ \int_{0}^{T} \int_{B} \bigl| \l\langle G_n(\m_n(s)) - G(\m(s)) , V \r\rangle_{L^2} \bigr|^2 \, \nu(dl) \, ds \r] = 0.
		\end{equation}
	\end{lemma}

		\begin{lemma}\label{lemma convergence of control term}
			Let $V\in L^6(\Omega^\p:L^2)$. Then the following convergence holds.
			\begin{align}
				\lim_{n\to\infty} \mathbb{E}^\p \l| \int_{0}^{T} \int_{\mathbb{U}} \l\langle L\l( \m_n(s) , v \r) , V \r\rangle_{L^2} \, \lambda_n^\p(dv,ds)
				- \int_{0}^{T} \int_{\mathbb{U}} \l\langle L\l( \m_n(s) , v \r) , V \r\rangle_{L^2} \, \lambda^\p(dv,ds) \r|.
			\end{align}
		\end{lemma}
		\begin{proof}[Proof of Lemma \ref{lemma convergence of control term}]
			We begin by first observing that
			\begin{align}\label{eqn convergence of the control term equation 1}
				\nonumber & \mathbb{E}^\p \int_{0}^{T}  \int_{\mathbb{U}} \l\langle L\l( \m_n(s) , v \r) , V \r\rangle_{L^2}\lambda_n^\p(dv,ds) 
				- \int_{0}^{T} \int_{\mathbb{U}} \l\langle L\l( \m(s) , v \r) , V \r\rangle_{L^2} \, \lambda^\p(dv,ds)  \\
				\nonumber = &  \mathbb{E}^\p \l( \int_{0}^{T} \int_{\mathbb{U}}  \l\langle L\l( \m_n(s) , v \r) , V \r\rangle_{L^2}\lambda_n^\p(dv,ds) 
				- \int_{0}^{T} \int_{\mathbb{U}} \l\langle L\l( \m(s) , v \r) , V \r\rangle_{L^2} \, \lambda_n^\p(dv,ds) \r) \\
				\nonumber & +  \mathbb{E}^\p \l( \int_{0}^{T} \int_{\mathbb{U}}  \l\langle L\l( \m(s) , v \r) , V \r\rangle_{L^2}\lambda_n^\p(dv,ds) 
				- \int_{0}^{T} \int_{\mathbb{U}} \l\langle L\l( \m(s) , v \r) , V \r\rangle_{L^2} \, \lambda^\p(dv,ds) \r) \\
				= & D_1 + D_2.
			\end{align}
			Let us introduce an auxiliary cut-off function $\psi_R$ for $R\geq 0$. For $v\in \mathbb{U}$, we define
			\begin{equation}
				\psi_R(v) = \begin{cases}
					& 1,\ \text{if}\ \kappa(\cdot,v) \leq R,\\
					& 0,\ \text{if}\ \kappa(\cdot,v) \geq 2R.
				\end{cases}
			\end{equation}
			\textbf{Calculations for $D_1$:}\\
   We first deal with the difference $D_1$. First, we note that the difference $D_1$ can be written as follows.

   \begin{align}
				\nonumber \mathbb{E}^\p  & \int_{0}^{T} \int_{\mathbb{U}}  \l\langle L\l( \m_n(s) , v \r) - L\l( \m(s) , v \r) , V \r\rangle_{L^2}\lambda_n^\p(dv,ds) \\
				\nonumber = & \mathbb{E}^\p  \int_{0}^{T} \int_{\mathbb{U}} \psi_R(v) \l\langle L\l( \m_n(s) , v \r) - L\l( \m(s) , v \r) , V \r\rangle_{L^2}\lambda_n^\p(dv,ds) \\
				\nonumber & + \mathbb{E}^\p  \int_{0}^{T} \int_{\mathbb{U}} \l( 1 - \psi_R(v) \r) \l\langle L\l( \m_n(s) , v \r) - L\l( \m(s) , v \r) , V \r\rangle_{L^2}\lambda_n^\p(dv,ds) \\
				= & \, D_{1,1} + D_{1,2}.
			\end{align}
   
			For $D_{1,1}$, we have the following sequence of inequalities. Let $K_{2R} = \l\{ v\in\mathbb{U} : \kappa(\cdot,v) \leq 2R\r\}$.
			\begin{align}
				\nonumber \l| D_{1,1} \r| \leq & \,  \mathbb{E}^\p \l| \int_{0}^{T} \int_{\mathbb{U}}  \psi_R(v) \l\langle L\l( \m_n(s) , v \r) - L\l( \m(s) , v \r) , V \r\rangle_{L^2}\lambda_n^\p(dv,ds) 
				 \r| \\
				\nonumber  \leq & \, C \, \mathbb{E}^\p \l| \int_{0}^{T}  \sup_{v\in K_{2R}} \l\langle  L\l( \m_n(s) , v \r) - L\l( \m(s) , v \r) , V \r\rangle_{L^2}
				 \r| \, ds\\
				\nonumber \leq & \, C \l( \mathbb{E}^\p  \int_{0}^{T}  \sup_{v\in K_{2R}} \l| L\l( \m_n(s) , v \r) - L\l( \m(s) \r) \r|_{L^2}^2 \, ds \r)^{\frac{1}{2}} 
				 \l(  \mathbb{E}^\p  \int_{0}^{T}  \l| V \r|_{L^2}^2  \, ds\r)^{\frac{1}{2}} \\
				 \leq & \, C \l( \mathbb{E}^\p  \int_{0}^{T}  \sup_{v\in K_{2R}} \l| L\l( \m_n(s) , v \r) - L\l( \m(s) \r) \r|_{L^2}^2 \, ds \r)^{\frac{1}{2}}. 
			\end{align}
			By Assumption \ref{assumption}, the right-hand side of the above inequality converges to $0$ as $n\to\infty$.

   	For the term $D_{1,2}$, we have the following argument.

   \begin{align}
				\nonumber \l| D_{1,2} \r| = & \l| \mathbb{E}^\p  \int_{0}^{T} \int_{\mathbb{U}} \l( 1 - \psi_R(v) \r) \l\langle L\l( \m_n(s) , v \r) - L\l( \m(s) , v \r) , V \r\rangle_{L^2}  \, \lambda_n^\p(dv,ds) \r| \\
				\nonumber \leq & \, C \, \mathbb{E}^\p  \int_{0}^{T} \int_{\mathbb{U}} \l| \l( 1 - \psi_R(v) \r) \r|  \l| L\l( \m_n(s) , v \r) - L\l( \m(s) , v \r) \r|_{L^2} \l| V \r|_{L^2}  \,  \lambda_n^\p(dv,ds) \\	
    \nonumber \leq & \, C \, \mathbb{E}^\p  \int_{0}^{T} \int_{\mathbb{U}} \l| \l( 1 - \psi_R(v) \r) \r| \l[ \l| \m_n(s) \r|_{L^2}^{r} + \l| \m(s) \r|_{L^2}^{r} \r] \kappa(s,v) \l| V \r|_{L^2}  \,  \lambda_n^\p(dv,ds) \\	
    \nonumber \leq & \, C \l( \mathbb{E}^\p  \int_{0}^{T} \l[ \l| \m_n(s) \r|_{L^2}^{2r} + \l| \m(s) \r|_{L^2}^{2r} \r] \, ds \r)^{\frac{1}{2}} 
    \l( \mathbb{E}^\p  \int_{0}^{T} \int_{\mathbb{U}}  \l| 1 - \psi_R(v) \r| \kappa^3(s,v) \,  \lambda_n^\p(dv,ds) \r)^{\frac{1}{3}} \centerdot \\
    \nonumber & \qquad \l( \mathbb{E}^\p  \int_{0}^{T}  \l| V \r|_{L^2}^{6}  \, ds \r)^{\frac{1}{6}} \\
    \nonumber \leq  & \, C \l( \mathbb{E}^\p  \int_{0}^{T} \int_{\mathbb{U}}  \l| 1 - \psi_R(v) \r| \kappa^3(s,v) \,  \lambda_n^\p(dv,ds) \r)^{\frac{1}{3}} \\
				\nonumber \leq & \, C  \l( \mathbb{E}^\p  \int_{0}^{T} \int_{v\in\mathbb{U}:\kappa(\cdot,v \geq R)} \frac{1}{\kappa(s,v)} \kappa(s,v)^4  \,  \lambda_n^\p(dv,ds) \r)^{\frac{1}{3}} \\
			    \nonumber \leq & \, C \frac{1}{R^{\frac{1}{3}}} \l( \mathbb{E}^\p  \int_{0}^{T} \int_{v\in\mathbb{U}:\kappa(\cdot,v \geq R)} \kappa^4(s,v)   \, \lambda_n^\p(dv,ds) \r)^{\frac{1}{4}} \\
				\leq & \frac{C}{R^{\frac{1}{3}}}.
			\end{align}
   Here, we have utilized the assumption $0\leq r < 2$ (from Assumption \ref{assumption}), which in turn implies $2r < 4$. Hence we can use the inequalities \eqref{eqn L 4 L4 bound mn prime} and \eqref{eqn L 4 L4 bound m prime}.
			The right-hand side of the above inequality converges to $0$ as $n\to\infty$. This concludes the calculations for the term $D_1$ from \eqref{eqn convergence of the control term equation 1}.\\
   \textbf{Calculations for $D_2$:}\\
   We first recall the difference term here.
   \begin{align}\label{eqn D21 and D22}
       \nonumber D_2 = & \, \mathbb{E}^\p \l( \int_{0}^{T} \int_{\mathbb{U}}  \l\langle L\l( \m(s) , v \r) , V \r\rangle_{L^2}\lambda_n^\p(dv,ds) 
				- \int_{0}^{T} \int_{\mathbb{U}} \l\langle L\l( \m(s) , v \r) , V \r\rangle_{L^2} \, \lambda^\p(dv,ds) \r)\\
    \nonumber = & \mathbb{E}^\p \bigg( \int_{0}^{T} \int_{\mathbb{U}} \l( 1 - \psi_R(v) \r)  \l\langle L\l( \m(s) , v \r) , V \r\rangle_{L^2}\lambda_n^\p(dv,ds) \\
	\nonumber & \quad  - \int_{0}^{T} \int_{\mathbb{U}}\l( 1 - \psi_R(v) \r) \l\langle L\l( \m(s) , v \r) , V \r\rangle_{L^2} \, \lambda^\p(dv,ds) \bigg) \\
    \nonumber & + \mathbb{E}^\p \bigg( \int_{0}^{T} \int_{\mathbb{U}} \psi_R(v) \l\langle L\l( \m(s) , v \r) , V \r\rangle_{L^2}\lambda_n^\p(dv,ds) \\
	\nonumber & \quad - \int_{0}^{T} \int_{\mathbb{U}} \psi_R(v) \l\langle L\l( \m(s) , v \r) , V \r\rangle_{L^2} \, \lambda^\p(dv,ds) \bigg)\\
    = & D_{2,1} + D_{2,2}.
   \end{align}
   The calculations for the term $D_{2,1}$ can be done similar to the calculations for the term $D_{1,2}$. Hence we skip the calculations here and proceed with the term $D_{2,2}$.

We have the following calculations for the term $D_{2,2}$.
   Let us first define an operator $f:[0,T] \times \mathbb{U} \to \mathbb{R}$, given by 
   \begin{align}\label{eqn convergence F claim}
       f(s,v) = \Psi_R(v)\l\langle L\l( \m(s) , v \r) , V \r\rangle_{L^2},s\in[0,T],v\in\mathbb{U}.
   \end{align}
   \textbf{Claim: } $f\in L^1(0,T:C_b(\mathbb{U})),\ \mathbb{P}^\p$-a.s.
   \begin{proof}[\textbf{Proof of the claim}]

\begin{align}
    \nonumber \mathbb{E}^\p \l| f \r|_{L^1(0,T:C_b(\mathbb{U}))} = & \, \mathbb{E}^\p\int_{0}^{T} 
    \sup_{v \in \mathbb{U}} 
    \l| f(s,v) \r| \, ds  \\
    \nonumber = & \, \mathbb{E}^\p \int_{0}^{T} 
    \sup_{v \in \mathbb{U}} 
    \l| \Psi_R(v)\l\langle L\l( \m(s) , v \r) , V \r\rangle_{L^2} \r| \, ds  \\
    \nonumber \leq & \, \mathbb{E}^\p \int_{0}^{T} 
    \sup_{v \in \mathbb{U}} 
    \l| \Psi_R(v) \r| \l| L\l( \m(s) , v \r) \r|_{L^2} \l| V \r|_{L^2} \, ds  \\
    \nonumber \leq & \, C \, \mathbb{E}^\p \int_{0}^{T} \sup_{v \in \mathbb{U}} \l| \Psi_R(v) \r| \kappa(s,v) \l|  \m(s) \r|_{L^2}^r \l| V \r|_{L^2} \, ds  \\
    \nonumber \leq & \, C \l( \mathbb{E}^\p \int_{0}^{T} \sup_{v \in \mathbb{U}} \l| \Psi_R(v) \r| \kappa(s,v)^4 \, ds \r)^{\frac{1}{4}}
    \l( \mathbb{E}^\p \int_{0}^{T}  \l|  \m(s) \r|_{L^2}^{2r}  \, ds \r)^{\frac{1}{2}} \centerdot\\
    \nonumber & \qquad \l( \mathbb{E}^\p \int_{0}^{T}  \l| V \r|_{L^2}^4 \, ds \r)^{\frac{1}{4}} \\
     \leq & \, C \l( \mathbb{E}^\p \int_{0}^{T} \sup_{v \in \mathbb{U}} \l| \Psi_R(v) \r| \kappa(s,v)^4 \, ds  \r)^{\frac{1}{4}} < \infty.
\end{align}
Hence, we have
\begin{equation}
    \l| f \r|_{L^1(0,T:C_b(\mathbb{U}))} < \infty , \ \mathbb{P}^\p-\text{a.s.}
\end{equation}
This concludes the proof of the claim.
   \end{proof}
   We recall the term $D_{2,2}$ from \eqref{eqn D21 and D22}.
   \begin{align}
    \nonumber D_{2,2} = & \, \mathbb{E}^\p \bigg[ \int_{0}^{T} \int_{\mathbb{U}} \psi_R(v) \l\langle L\l( \m(s) , v \r) , V \r\rangle_{L^2}\lambda_n^\p(dv,ds) \\
	& \quad - \int_{0}^{T} \int_{\mathbb{U}} \psi_R(v) \l\langle L\l( \m(s) , v \r) , V \r\rangle_{L^2} \, \lambda^\p(dv,ds) \biggr].
   \end{align}
   Using Theorem 2.16 in \cite{ZB+RS}, along with the result \eqref{eqn convergence F claim}, we have the following convergence.
   \begin{align}
       \nonumber \lim_{n\to\infty} & \int_{0}^{T} \int_{\mathbb{U}} \psi_R(v) \l\langle L\l( \m(s) , v \r) , V \r\rangle_{L^2}\lambda_n^\p(dv,ds) \\
       & = \int_{0}^{T} \int_{\mathbb{U}} \psi_R(v) \l\langle L\l( \m(s) , v \r) , V \r\rangle_{L^2}\lambda^\p(dv,ds),\ \mathbb{P}^\p-\text{a.s.}
   \end{align}
   For $0\leq r < 2$, we can further prove that for some $q > 1$,
   
   \begin{equation}
       \mathbb{E}^\p \l| \int_{0}^{T} \int_{\mathbb{U}} \psi_R(v) \l\langle L\l( \m(s) , v \r) , V \r\rangle_{L^2}\lambda_n^\p(dv,ds) \r|^{q} < \infty,
   \end{equation}   
   and
    \begin{equation}
       \mathbb{E}^\p \l| \int_{0}^{T} \int_{\mathbb{U}} \psi_R(v) \l\langle L\l( \m(s) , v \r) , V \r\rangle_{L^2}\lambda^\p(dv,ds) \r|^{q} < \infty.
   \end{equation}
   Hence, using the Vitali convergence theorem (see for example, Theorem 4.5.4 in \cite{BogachevBook_MeasureTheor_2007}), we can conclude the convergence
   \begin{align}
       \nonumber & \lim_{n\to\infty} \mathbb{E}^\p \bigg[ \int_{0}^{T} \int_{\mathbb{U}} \psi_R(v) \l\langle L\l( \m(s) , v \r) , V \r\rangle_{L^2}\lambda_n^\p(dv,ds) \\
	& \quad - \int_{0}^{T} \int_{\mathbb{U}} \psi_R(v) \l\langle L\l( \m(s) , v \r) , V \r\rangle_{L^2} \, \lambda^\p(dv,ds) \biggr] = 0.
   \end{align}
   This concludes the calculations for $D_{2,2}$, and hence also the proof of Lemma \ref{lemma convergence of control term}.
		\end{proof}
Following the line of arguments from \cite{ZB+UM_WeakSolutionSLLGE_JumpNoise,ZB+UM+Zhai_Preprint_LDP_LLGE_JumpNoise,UM+AAP_2021_LargeDeviationsSNSELevyNoise}, we can prove that the tuple 
$$\pi^* = \l( \Omega^\p , \mathcal{F}^\p , \mathbb{F}^\p , \mathbb{P}^\p , \m , \eta^\p , \lambda^\p \r),$$
is a weak martingale solution to the problem \eqref{eqn relaxed controlled problem considered Marcus Form}.
What remains now is to show that the obtained solution $\pi^*$  indeed minimizes the cost $\mathscr{J}$.
By Lemma \ref{lemma use of Skorohod Theorem}, we observe that the laws of $\l(\m_n,\lambda_n^\p\r)_{n\in\mathbb{N}}$ and $\l(m_n,\lambda_n\r)_{n\in\mathbb{N}}$ are equal, and therefore,
\begin{align}
   \liminf_{n\to\infty} \mathbb{E}^\p \int_{0}^{T} F(\m_n(s),v) \, \lambda_n^\p(dv,ds)
    \leq  \liminf_{n\to\infty} \mathbb{E} \int_{0}^{T} F(m_n(s),v) \, \lambda_n(dv,ds).
\end{align}
By the convergence \eqref{eqn L4L4L4 convergence mn prime} of the sequence $\l\{ \m_n \r\}_{n\in\mathbb{N}}$, along with the lower semicontinuity assumption on the running cost $F$ (see Assumption \ref{assumption}) and followed by the Fatou's Lemma (see \cite{Royden_1968Book_RealAnalysis}), we conclude that
\begin{align}
    \nonumber \mathscr{J}(\pi^*) = & \, \mathbb{E}^\p \int_{0}^{T} F(\m(s),v) \, \lambda^\p(dv,ds) \\
    \nonumber \leq & \liminf_{n\to\infty} \mathbb{E}^\p \int_{0}^{T} F(\m_n(s),v) \, \lambda_n^\p(dv,ds) \\
    \leq & \liminf_{n\to\infty} \mathbb{E} \int_{0}^{T} F(m_n(s),v) \, \lambda_n(dv,ds) = \Lambda.
\end{align}
The last equality holds because the sequence $\l\{\pi_n\r\}_{n\in\mathbb{N}}$ is a minimizing sequence (see the equality \eqref{eqn pi n is a minimizing sequence}).
Hence $\pi^*$ minimizes the cost $\mathscr{J}$, and hence is an optimal control, as per Definition \ref{definition optimal control}, for the problem \eqref{eqn relaxed controlled problem considered Marcus Form}. This concludes the proof of Theorem \ref{thm Optimal control}.

\dela{

 \coma{Possibly remove the statement of this Theorem. Reference should suffice.}
		\begin{theorem}[Skorohod-Jakubowski, \cite{ZB+UM_WeakSolutionSLLGE_JumpNoise} ]\label{THeorem Skorohod Jakubowski}
		Let $\mathcal{X}$ be a topological space such that there exists a sequence of continuous functions $f_m : \mathcal{X}\to\mathbb{C}$ that separates points of $\mathcal{X}$. Let $\mathcal{V}$ be the $\sigma$-algebra generated by $\l\{ f_m \r\}_{m\in\mathbb{N}}$. Then we have the following assertions.
		\begin{enumerate}
			\item Every compact set $\mathcal{K}\subset \mathcal{X}$ is metrizable.
			\item Let $\mu_n$ be a tight sequence of probability measures on $\l( \mathcal{X} , \mathcal{V} \r)$. Then there exists a subsequence $\l\{\mu_{n_{k}}\r\}_{k\in\mathbb{N}}$, random variables $X_{k},X$ for $k\in\mathbb{N}$ on a common probability space $\l( \bar{\Omega} , \bar{\mathbb{F}} , \bar{\mathbb{P}}\r)$ with $\bar{\mathbb{P}}^{X_{k}} = \mu_k$ for each $k\in\mathbb{N}$ and $X_k\to X$ $\bar{\mathbb{P}}$-a.s. as $k\to\infty$.
		\end{enumerate}
	\end{theorem}

 }
	
	\bibliographystyle{plain}
	\bibliography{References_GokhaleSoham}
\end{document}